\documentclass[a4paper,11pt,reqno]{amsart}

\usepackage{
amsfonts,
amsmath,
amsopn,
amssymb,
amsthm,
bbm,
bbold,
dsfont,
enumitem,
graphicx,
mathrsfs,
mathtools,
mathabx,
soul,
subfig,
verbatim,
xcolor,
xspace,
latexsym
}

\usepackage{geometry}
\usepackage[hidelinks]{hyperref}
\usepackage[utf8]{inputenc}

\usepackage{rotating}

\usepackage{tikz}
\usetikzlibrary{decorations.pathreplacing}
\usetikzlibrary{shapes.geometric,calc,decorations.markings,math}
\tikzstyle{nodo}=[circle,draw,fill,inner sep=0pt,minimum size=%
1.5mm]
\tikzstyle{infinito}=[circle,inner sep=0pt,minimum size=0mm]

\mathtoolsset{showonlyrefs}

\newtheorem{theorem}{Theorem}[section]

\newtheorem{proposition}[theorem]{Proposition}
\newtheorem{corollary}[theorem]{Corollary}

\theoremstyle{remark}
\newtheorem{remark}[theorem]{Remark}
\newtheorem*{remark*}{Remark}

\theoremstyle{definition}
\newtheorem{definition}[theorem]{Definition}

\newcommand{\R}{\mathbb{R}}

\newcommand{\C}{\mathbb{C}}

\newcommand{\J}{\mathcal H}
\newcommand{\ju} {\mathcal J}
\newcommand{\D}{\mathcal{D}}

\newcommand{\I}{\mathcal{I}}

\newcommand{\E}{\mathcal{E}}

\renewcommand{\Re}{\mathrm{Re}}
\renewcommand{\leq}{\leqslant}
\renewcommand{\geq}{\geqslant}

\newcommand{\x}{\mathbf{x}}

\newcommand{\ov}[1]{\overline{#1}}

\title[NLS on a hybrid domain]{The search for NLS ground states on a hybrid domain: \\ Motivations, methods, and results. }

\author[R. Adami]{Riccardo Adami}
\address{R. ADAMI: Politecnico di Torino, Dipartimento di Scienze Matematiche ``G.L. Lagrange'', Corso Duca degli Abruzzi, 24, 10129, Torino, Italy}
\email{riccardo.adami@polito.it}

\author[F. Boni]{Filippo Boni}
\address{F. BONI: Scuola Superiore Meridionale, Largo S. Marcellino, 10, 80138, Napoli, Italy}
\email{f.boni@ssmeridionale.it}

\author[R. Carlone]{Raffaele Carlone}
\address{R. CARLONE: Università degli Studi di Napoli Federico II, Dipartimento di Matematica e Applicazioni ``Renato Caccioppoli”, Via Cintia, Monte S. Angelo, 80126, Napoli, Italy}
\email{raffaele.carlone@unina.it}

\author[L. Tentarelli]{Lorenzo Tentarelli}
\address{L. TENTARELLI: Politecnico di Torino, Dipartimento di Scienze Matematiche ``G.L. Lagrange'', Corso Duca degli Abruzzi, 24, 10129, Torino, Italy}
\email{lorenzo.tentarelli@polito.it}

\date{\today}

\begin{document}


\begin{abstract}
We discuss the problem of establishing the existence of the Ground States for the subcritical focusing Nonlinear Schr\"odinger energy on a domain made of a line and a plane intersecting at a point. The problem is physically motivated by the experimental
realization of hybrid traps for Bose-Einstein Condensates, that are able
to concentrate the system on structures close to the domain we consider. In fact, such a domain approximates the trap as the temperature approaches the absolute zero. The spirit of the paper is mainly pedagogical, so we focus on the formulation of the problem and on the explanation of the result, giving references
for the technical points and for the proofs.
\end{abstract}

\maketitle

\vspace{-.5cm}

\noindent {\footnotesize \textul{Keywords:} hybrids, standing waves, nonlinear Schr\"odinger, ground states, delta interaction.}


\section{Introduction}
Modeling Bose-Einstein Condensation  has been a major challenge in the
Mathematical Physics of the last decades. The phenomenon consists in a phase transition undergone by systems made of a large number, i.e. around $10^4$, of identical bosons at low temperature, namely $10^{-7}$ K and below. In the new phase all bosons of the system share the same quantum state called Ground State of the condensate, and such state occupies all the available space. Notice that the fact that every boson possesses a quantum state is not a 
natural property: indeed, even assuming that the whole system lies in a quantum state,  its parts may not. In the mathematical description, this means that a wave function of the whole system may not be the product of single-particle wave functions. This is what happens in the presence of entanglement among the particles. Therefore, in the transition to a Bose-Eintein condensate every particle 
acquires an individual quantum state, i.e. 
bosons in a Bose-Einstein condensate lying in its Ground State  are  not entangled to one another.

Such a new physical phase was  first foreseen one hundred years ago \cite{B-24,E-24}
and experimentally realized seventy years later
\cite{CW-95,K-95}, when the techinques of laser and 
evaporative cooling made  possible to reach the necessary
temperature. Since then, the mathematical investigation on condensates has become topical and
has involved several approaches: operator theory, nonlinear evolution equation, calculus of
variation, analysis on Fock spaces and the formalism of second quantization. An important part of
such a research consisted in singling out the conditions under which the so-called
Gross-Pitaevskii regime (\cite{B-47,G-61,P-63}) holds. In such a regime it is 
possible to approximate the fundamental linear $N$-body dynamics of the bosons with an
effective, nonlinear one-body dynamics: if, on the one hand, the problem becomes more difficult
for the presence of the nonlinearity, on the other hand there is a great simplification in passing
from an $N$-body to a one-body problem, and the simplification greatly overcomes the
additional difficulty. 

Let us be more specific. First, a quantum mechanical description of
the $N$-boson system is made through the $N$-boson wave function $\Psi_{N,t} (x_1, \dots x_N)$, that represents
the quantum state of the system at time $t$. By
$x_j$ we denote the three-dimensional coordinate of the $j$th boson. The function $\Psi_{N,t}$ is symmetric under permutation of its variables, as it describes the state of $N$
identical bosons. The evolution of the system is described by the $N$-body Schr\"odinger
equation
\begin{equation}
    \label{Nbody}
i \partial_t \Psi_{N,t} \ = \ - \sum_{j=1}^N \Delta_j 
\Psi_{N,t} + \sum_{1 \leq i < j \leq N} w (x_i - x_j ) 
\Psi_{N,t} + \sum_{j=1}^N V (x_j)
\Psi_{N,t},
\end{equation}
where $w$ represents the two-body interaction between the bosons and $V$ is the potential 
of an external field whose role is typically to confine the system.

The already mentioned Gross-Pitaevskii regime is the physical setting in which 
the dynamics of a Bose-Einstein condensate is effectively described by the one-body nonlinear Gross-Pitaevskii equation 
\begin{equation} \label{gptrap}
i \partial_t \psi_t = - \Delta \psi_t + V(x) \psi_t + 8 \pi \alpha | \psi_t |^2 \psi_t,
\end{equation}
where
$\alpha$ is the scattering length of the two-body interaction $w$ between the particles in the condensate. It turns then out that the nonlinearity is the mark of the interaction between the bosons that are the elementary constituents of the condensate. Of course, the problem of
deducing \eqref{gptrap} from \eqref{Nbody} 
is of paramount importance in understanding the underlying physics. It is nowadays
well understood that the transition to the Gross-Pitaevskii regime happens if the interacting
potential scales with respect to the number $N$ of bosons as
\begin{equation}
    \label{scaling}
w (x) \to N^2 w (Nx)
\end{equation}
and if the initial data fulfil at least approximately a bosonic Stosszahlansatz, namely $\Psi_{N,0} = \psi_0^{\otimes N}$. The scaling law \eqref{scaling}
means that the Gross-Pitaevskii regime is a good approximation of the fundamental dynamics \eqref{Nbody} if the interacting potential is strong, due to the factor 
$N^2$, and concentrated, as expressed by the argument $Nx$ that
shrinks the range of the potential to a scale $N^{-1}$. The Stosszahlansatz hypothesis states in turn that the correlations between 
the particles are small, which is more likely the case if
the system is dilute. Notice that at the level of the dynamics
the process that makes the system uncorrelated is not described, as one has to impose the Stosszahlansatz on the initial data. Under such hypothesis, one has that in the 
Gross-Pitaevskii regime
every particle possesses an individual state, so 
particles are not entangled with one another, and that at every time $t$ such a state is represented by a 
complex-valued function $\psi_t$ to be interpreted according the usual Born's rule: the modulus square of  $\psi_t$ gives the probability density of finding one particle of the condensate in a specified spatial region. Such interpretation imposes $\psi_t$ to be normalized at $1$, but in fact it is often preferable to normalize
it to the number $N$ of particles in the condensate,
so that
$$ \int_{\R^3} | \psi_t|^2 \, dx \ = \ N ,$$
thus $| \psi_t |^2$ is better understood as a particle density rather than as a probability density.


Besides their relevance in the physical
understanding of the phenomenon,
the rigorous proofs 
of Bose-Einstein condensation in the static framework (\cite{SeiringerMerda}) and of the stability of the condensates under time evolution (\cite{AGT-07,ESY-07,ESY-10,KSS-11}) greatly
enriched the mathematical techniques employed in the derivation of nonlinear one-body equations from linear $N$-body  quantum dynamics of interacting particles. In particular,
the scaling \eqref{scaling} is highly singular, 
 so
a main breakthrough was accomplished as it was understood
how to rigorously deal with it.

Later, a major challenge was to get a good estimate
of the error made in approximating the fundamental linear
$N$-body description with the effective nonlinear one-body evolution equation (\cite{RS-07,KP-10,BdOS-15,NN-17,BBCS-20}).

In the following we focus on the existence and  the shape of the Ground State for a condensate in
the Gross-Pitaevskii regime under the action of a 
trap whose shape can be modeled as the union of a 
plane and a line orthogonal to it. 
Since the Bose-Einstein condensation occurs at low energy, it does not involve
neither the creation nor the annihilation of particles, therefore
the natural notion of Ground State is that of the minimizer of the energy functional $E_{GP}$ whose value is
conserved by the flow of the Gross-Pitaevskii equation, under the
constraint that the number of particles, i.e. the $L^2$-norm, is
fixed. In the following the $L^2$-norm will be called mass and will be denoted  by $\mu$.  We will refer to the prescription that it has to be a fixed quantity as to the mass constraint.

We then get to the problem of finding the minimizers of the Gross-Pitaevskii energy
\begin{equation} \label{energycondensate}
E_{GP} (\psi) \ = \ \frac 1 2 \| \nabla \psi \|_{L^2(\R^3)}^2 + 2 \pi \alpha
\| \psi \|_{L^4(\R^3)}^4 + \frac 12 \int_{\R^3} V(x) | \psi (x) |^2 \, dx
\end{equation}
with the constraint
$$ \int_{\R^3} |\psi (x) |^2 \, dx = \mu,$$
provided they exist.

Usually the potential $V$ is referred to as the {trap}. Of course, the actual shape of a condensate depends on its particular state, thus
 it has become customary to think of the spatial distribution of the Ground State as the shape of the trap.

Early condensates (\cite{CW-95,K-95}) were subject to an external harmonic and isotropic potential, so that their shape was approximately spherical, but the subsequent impressive technological advances in magnetic and optical confinement made it possible to build up condensates with various shapes: disc-shaped, cigar-shaped, branched, and others.

We assume that the trapping potential $V$ is so effective that the Ground State of the condensate can be approximated as supported in a region $\Omega \subset \R^3$, so that
\begin{equation} \label{constrainedenergy}
E (\psi) \ = \ \frac 1 2 \| \nabla \psi \|_{L^2(\Omega)}^2+ 2 \pi \alpha
\| \psi \|_{L^4(\Omega)}^4.
\end{equation}
Clearly, the Ground State $\psi_t$
evolves in time 
just by periodically changing its phase, namely
$$
\psi_t (x) \ = \ e^{i \omega t} \psi_0 (x),
$$
so that its spatial profile and therefore the associated density
of particles remain unaltered. It is a stationary state.

\subsection{Hybrid traps}
As already mentioned, we are interested in describing an experimental arrangement in which the trap has the shape of a plane with a line attached to it, as in Fig. \ref{fig-hyb}. This structure results from 
gluing together two components of different dimensionality and will be referred to as the hybrid structure $\J$. Traps of this kind
have been realized by employing a magnetic trap together with an optical trap (\cite{MFVK-15}). Because of the use of two different physical mechanisms, in the physical literature such traps are called hybrid, while for us the same word refers to the resulting multi-dimensional geometry of the system. Serendipitously,   it happens that hybrid traps give rise to hybrid geometry,
although
the two terms were introduced independently with different meaning.

\begin{figure} 
\centering
\begin{tikzpicture}[xscale= 0.5,yscale=0.5]
\node at (-3,-2) [infinito] (-3-2) {};
\node at (9,-2) [infinito] (9-2) {};
\node at (-4,0) [infinito] (-40) {};
\node at (-1,0) [infinito] (-10) {};
\node at (11,0) [infinito] (110) {};
\node at (14,0) [infinito] (140) {};
\node at (8,3) [nodo] (83) {};
\node at (8,13) [infinito] (813) {};
\node at (8,16) [infinito] (816) {};
\node at (8,0) [infinito] (80) {};
\node at (8,-7) [infinito] (8-7) {};
\node at (8,-10) [infinito] (8-10) {};
\node at (2,6) [infinito] (26) {};
\node at (5,6) [infinito] (56) {};
\node at (7,8) [infinito] (78) {};
\node at (17,6) [infinito] (176) {};
\node at (20,6) [infinito] (206) {};
\node at (19,8) [infinito] (198) {};

\node at (1,4) [infinito] (14) {$\J$};
\node at (3.5,1) [infinito] (41)  {$\Pi$};
\node at (9,3) [infinito] (8525) {$\mathcal J$};
\node at (7,12) [infinito] (712) {\begin{turn}{90} 
$\ell$ 
\end{turn}};
\draw [-] (-10) -- (110);
\draw [-] (-10) -- (56);
\draw [-] (110) -- (176);
\draw [-] (56) -- (176);
\draw [-] (83) -- (813);
\draw [-] (80) -- (8-7);
\draw [dashed] (-3-2) -- (-10);
\draw [dashed] (-40) -- (-10);
\draw [dashed] (110) -- (140);
\draw [dashed] (9-2) -- (110);
\draw [dashed] (26) -- (56);
\draw [dashed] (176) -- (198);
\draw [dashed] (56) -- (78);
\draw [dashed] (176) -- (206);
\draw [dashed] (813) -- (816);
\draw [dashed] (83) -- (80);
\draw [dashed] (8-7) -- (8-10);
\end{tikzpicture}
\caption{ The hybrid $\J$. The origin of the coordinates on both the line $\ell$ and the plane $\Pi$ is set at the junction $\mathcal J$.}
\label{fig-hyb}
\end{figure}
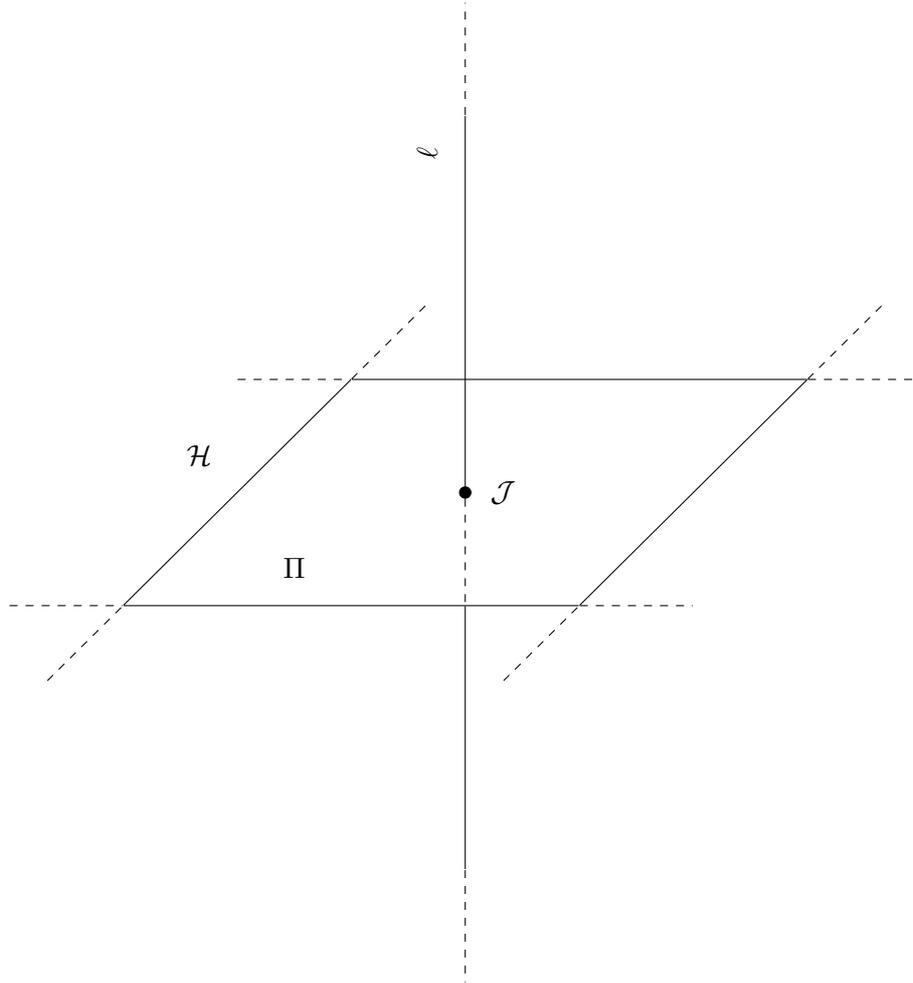

Let us describe more closely the effect of a hybrid trap. The confining potential $V$ decomposes into a term $V_M$, describing the magnetic trap, and a term $V_O$, related to the optical trap, so that
$$ V(\x) = V_M (\x) + V_O (\x), \qquad \x = (x,y,z) \in \R^3, $$
where, according to the description given in \cite{MFVK-15}, the effective
shape of the two potentials can be modeled as
$$V_M (\x) = 2 \nu B |z|, \qquad V_O = \frac 12 m \omega^2_O (x^2 + y^2).$$
Here $\nu$ is the magnetic moment of the particles in the condensate and $m$ their mass, $B$ is the magnetic field generated by the trap, and
$\omega_O$ is the frequancy of the optical trap. The net effect of the double trap is that the resulting spatial distribution of the condensate is
\begin{equation} \label{hybriddensity}
n (\x) \ = \ n_0 \, e^{- \frac{2 \nu B }{kT} |z|} e^{- \frac
{m \omega^2_O}{2 kT} (x^2+ y^2)}, 
\end{equation}
where $k$ is the Boltzmann constant and $T$ is the temperature of the system. Notice that the first factor concentrates the condensate on the plane $z = 0$, while the second factor on the $z$ axis.

One can then distinguish three regions in the hybrid, according to the size of the density: 

\begin{enumerate}
\item
the origin, where both
exponentials in  
\eqref{hybriddensity} equal one, so the density is of order $1$;
\item the $(x,y)$-plane, where the density $n$ is dumped by the optical trap, namely by the second exponential factor in \eqref{hybriddensity}, and the $z$-axis, where the first exponential factor, representing the magnetic trap, attenuates the density of particles;
\item the rest of the space, where the density is shrunk by both traps.
\end{enumerate}
Moreover, notice that as the temperature goes to the absolute zero, the density concentrates at the origin, with an exponential 
tail on $\J$, and a double exponentially tiny rest outside $\J$. Our aim is to describe such profile at least qualitatively.

\medskip
Hybrids were first introduced in \cite{ES-87}, then further investigated in \cite{ES-88} and later in \cite{CE-11,CP-17}. In such works the physical context ranges from classical electromagnetism to quantum
mechanics, and the dynamics is always linear. The hybrid structure considered there is made of a plane with a halfline attached  at its origin. 
On the same hybrid manifold we studied the problem of the Ground State of condensates in \cite{ABCT-24}. The problem we investigate here is a simple variant of that, so we refer to that work for the technical aspects of the proofs. Here we aim at discussing the meaning and the 
formulation of the problem. Therefore, in Section 2 we construct the energy functional
on $\J$ and compare it to the classical Nonlinear Schr\"odinger energies. In Section 3 we
give the Gargliardo-Nirenberg estimates, that are crucial for the existence of a Ground State,
and present the main result. Technically it deals with the lack of compactness of minimizing
sequences, but it can be told quite simply, by saying that a Ground State exists 
if the hybrid is energetically more convenient than the line. In other words, the hybrid must
compete, at an energetical level, with the line and {\em not} with the plane. Finally,
we give sufficient conditions to make the hybrid energetically convenient, namely, conditions
on the parameters of the system in order to get this convenient. 

One can sketch such conditions as follows. If one aims at constructing an effective trap for a Bose-Einstein condensate with a shape close to $\J$,
namely a trap able to capture the system in a geometry of that kind, then  the junction must unfold an attractive effect on the condensate. Now, since the most general junction can produce a contact interaction on the line, another contact interaction on the plane, and a coupling between the two, one has to consider the possible attractive feature of the three of them. More precisely, to trap a condensate on
a region modeled by $\J$, one must ensure that:
\begin{enumerate}
    \item the junction between the line and the plane deploys an attractive action on the part of the condensate located on the line, towards the origin of the line;
    indeed, this makes the system \textquotedblleft line + junction" energetically convenient with respect
    to the line alone, and so the hybrid prevails on the line;
    \item if the previous condition is not fulfilled, then one has to make the plane more convenient than the line, so the hybrid will be more convenient too;
    \item if both the previous conditions are not fulfilled, then one has to provide that the coupling 
    between the plane and the line is strong: indeed its effect always proves  energetically convenient and, above a certain threshold, it makes the hybrid more convenient than the line.
\end{enumerate}
These prescriptions are stated more precisely in Corollary \ref{thm:convenient}.

On the other hand, if none among the conditions
(1), (2), and (3) is satisfied, then it may happen that
the condensate tends to escape through the line far from the junction.
According to our analysis, this runaway phenomenon cannot
take place on the plane.

We stress here that in our model we are not able to 
treat the fourth power nonlinearity for the component of the energy on the plane, for technical reasons that will be clarified in the next section. We can of course approach such power
arbitrarily, but this does not guarantee that our result
can be transposed by continuity to the physical case with the power four.

Propositions \ref{prop:e2q} and \ref{massconserved} are proved in
some detail since their direct proof is not in \cite{ABCT-24}.

\section{The problem}
We assume that the spatial domain in which the condensate is confined by the trap is  represented by the hybrid domain $\J$ and
aim at investigating whether, given the number of particles $N$ that compose the condensate, there is a Ground State. Let us translate such problem in the terms of the Calculus of Variations.

\subsection{Heuristic construction of the energy functional}
First, we write the hybrid domain $\J$ as
$$ \J \ = \ \ell \cup \Pi, $$ 
where $\ell$ is the $z$-axis and $\Pi$ is the $(x,y)$-plane.
Second, we introduce a mathematical object that plays
the role of a wave function defined on $\J$. It is natural to assume that the state of the condensate is represented by
a two-component function $U = (u,v)$,
 where $u : \ell \to \C$ is the portion of the wave function supported on the line, while $v: \Pi \to \C$ is the part supported on the plane. In order to establish to which spaces the 
functions $u$ and $v$ belong, 
we have to introduce an energy functional $E$ that adapts the energy \eqref{energycondensate} to the structure of $\J$. It is natural to assume that such energy functional is composed by three terms:
$$E (U) = E_\alpha (u, \ell) + E_\sigma (v, \Pi) + E_{\ju} (U), $$
where
\begin{itemize}
\item $E_\alpha (\cdot, \ell)$ is the energy of the portion of $U$ on $\ell$. It comprises the kinetic energy of the one-dimensional wave packet $u$, the self-consistent nonlinear term, and the contact interaction between the line and the junction.
Then, we write it as
\begin{equation} \label{ealfa}
E_\alpha (u, \ell) \ = \ \frac 12 \| u' \|_{L^2 (\R)}^2 - \frac 1 p 
\| u \|_p^p + \frac \alpha 2 |u (0)|^2, \qquad 2 < p < 6, \, 
\alpha \in \R.
\end{equation}
Notice that the contact interaction corresponds here to
the action of a Dirac's delta potential.

\item $E_\sigma (v, \Pi)$ is the energy of the portion of $U$ on $\Pi.$ It includes a nonlinear term and a contact interaction at the junction. Now, as widely known 
(\cite{AGHKH-88,RS-80}), a pointwise contact interaction in dimension two cannot be described by a distributional delta potential as in dimension one, i.e. it cannot be written as $| v (0) |^2$, since in two dimensions a generic function in $H^1$ has no pointwise evaluation. More precisely, contrarily to the one-dimensional case, in two dimensions such a pointwise perturbation is not small with respect to the kinetic term in the sense of the quadratic forms, therefore the definition of the point interaction cannot be 
achieved in the same way. It is nowadays well understood that the construction of such kind of interaction can be equivalently  realized by following two different paths:
a renormalization procedure \cite{BF-61} or the use of the theory of self-adjoint extensions of hermitian operators \cite{RSII-80}. Even though the two approaches end up with the same result, they are complementary to each other, as the renormalization procedure highlights the interpretation of contact interactions as limits of spatially extended potentials, while the self-adjoint extension theory provides
a simple algorithm that directly gives all possible 
point interactions \cite{AGHKH-88}.

Eventually the introduction of a contact interaction modifies the energy domain by introducing a singularity at the site of the interaction. More specifically such a domain results in
\begin{equation} \label{energydom2d}
\D_2 = \left\{ u = \phi + q \frac{K_0}{2 \pi}, \, \phi \in H^1 (\R^2), \, q \in \C
\right\},
\end{equation}
where $K_0$ is a  function of the Bessel family, more precisely the MacDonald function of order zero whose asymptotic behaviour at the origin is $K_0 (x) \sim -   \log |x|$ and that decays exponentially fast
for large $|x|$ \cite{GR-66}. 
Notice that the presence of a point interaction in dimension two results in the formation
of a logarithmic singularity.
The complex factor $q$ is usually called the charge and 
represents the size of the singularity of the function $v$. Notice that this prescription  concentrates a significant portion 
of the state around the origin unless $q = 0$. 

The energy functional $E_\sigma (\cdot, \R^2)$ acts as follows:
\begin{equation} \begin{split}
\label{esigma}
E_\sigma (v, \Pi) \ = \ & \frac 1 2 \| \phi \|_{H^1 (\R^2)}^2 - \frac 1 2
\| v \|_{L^2 (\R^2)}^2 \\ & + \frac \sigma 2 |q|^2 -
\frac 1 r
\| v \|_{L^r (\R^2)}^r, \quad 2 < r < 4, \, \sigma \in \R.
\end{split}
\end{equation}
We stress that here the coupling constant $\sigma$ does not coincide with its analogous 
$\rho$ employed in \cite{ABCT-24} as well as with the 
parameter $\sigma$ previously used in \cite{ABCT-22}. Indeed,
one has
$$ \sigma = \rho + \frac {\gamma - \log 2} {2 \pi},$$
where $\gamma$ is the Euler-Mascheroni constant.

\item $E_{\ju} (U)$ is the energy of the interaction between $\ell$ and $\Pi$ that takes place at the junction. It describes exclusively the exchange of energy between the line and the plane. In other words, it is the only term which, under a dynamical point of view, realizes the transmission of a signal from $\ell$ to $\Pi$ and vice versa. For its precise shape we choose the simplest possible option, namely a quadratic coupling:
\begin{equation} \label{ebeta}
E_{\ju} (U) \ = \ - \beta \Re ( \overline{q} u(0) ), \qquad \beta \geq 0.
\end{equation}
\end{itemize} 

Summarizing, we constructed the energy functional

\begin{equation}
    \label{e} \begin{split}
E (U)  =  & \,E_\alpha (u, \ell) + E_\sigma (v, \Pi) 
+ E_{\ju} (U) \\
 =  &
 \,\frac 1 2 \| u' \|_{L^2 (\R)}^2 + \frac \alpha 2 | u (0) |^2 - \frac 1 p \| u \|^p_{L^p (\R)}
+ \frac 1 2 \| \phi \|_{H^1 (\R^2)}^2 
- \frac 12 \|v \|_{L^2 (\R^2)}^2  \\ &
+
\sigma
\frac{|q|^2}{2} - \frac 1 r \| v \|^r_{L^r (\R^2)}
-  \beta \Re (\ov q {u(0)}), \\ & \, {\rm{with}}
\  2 < p < 6, \ 2 < r < 4, \ \alpha, \sigma \in \R,
 \beta \geq 0,
\end{split}
\end{equation}
defined on the energy domain
\begin{equation} \label{d} \begin{split}
{\mathcal D} : = & \left\{ U = (u,v) \, {\rm s.t.} \, u \in H^1 (\R), \, 
v = \phi + q \frac {K_0} {2 \pi}, \,
\phi \in H^1 (\R^2), \, q \in \C \right\}.
\end{split}
\end{equation}

\subsection{Some comments on the energy functional $E$} \label{sec:comments}
Let us explain the meaning of the choice of the parameters in the construction of the energy functional \eqref{e}.


Concerning the term $E_\alpha (u, \ell)$
notice that in \eqref{ealfa} we consider an arbitrary nonlinearity power $p$ and limit our analysis to the focusing case, namely the case in which the nonlinearity has a negative sign and thus, dynamically, an attractive effect. The upper bound $p < 6$ means that we restrict to nonlinearities that can be controlled by the kinetic term, as described by the one-dimensional 
Gagliardo-Nirenberg estimate
\begin{equation} \label{gn1}
\| u \|_{L^p (\R)}^p \ \leq \ C \| u' \|_{L^2 (\R)}^{\frac p 2 -1}
\| u \|_{L^2 (\R)}^{\frac p 2 + 1}, \qquad u \in H^1 (\R),
\end{equation}
showing that, while for $p < 6$ the nonlinearity grows
slower than the kinetic term, for $p = 6$ it may grow at the same rate of the kinetic term, making the energetic balance more delicate. 
Eventually, it is well-known \cite{C-03} that if $\alpha = 0$ then the choice $p = 6$ results in a phenomenology that radically differs from that of the case $p < 6$, both in the statics and in the dynamics: if $p=6$ then existence of Ground States depends on the chosen value of $\mu$, and in the related evolution problem blow-up solutions appear whose existence 
cannot be extended to arbitrarily large time. 
Both features are not present if $p<6$, for which a Ground State
for $E_0 (u, \ell)$ is present at every value of the mass and all solutions to the dynamical problem in the energy space are defined globally in time.

For its peculiar character,
the power nonlinearity with $p= 6$ is called critical, and here we
treat subcritical nonlinearities only.

Let us point out that for $\alpha = 0$ the resulting energy 
functional $E_0 (\cdot, \ell)$ is the standard NLS energy functional with focusing subcritical power nonlinearity. It is
a long standing result \cite{ZS-71,C-03} that for every value of the mass there is a
family of ground states, called solitons, that are obtained by translating and multiplying by a phase the unique positive, even
Ground State
\begin{equation} \label{soliton}
\varphi_\mu (x) \ = \  \mu^{\frac 2 {6-p}} \varphi ( \mu^{\frac
{p-2}{6-p}} x),
\end{equation}
where
$$\varphi (x)= C_p \, {\rm sech}^{\frac 2 {p-2}} (c_p x),$$
 with $C_p$ and $c_p$  constants depending on the power $p$ only.
 
Furthermore, the energy of the solitons scales as
$$ \mathcal E_{0, \ell} (\mu) \ = \ E_0 (\varphi_\mu, \ell) \ =
\ - \theta_p \mu^{\frac{p+2}{6-p}}, $$
where $\theta_p$ is a positive constant that depends on $p$ only,
and we  introduced the symbol
\begin{equation}
\mathcal E_{\alpha, \ell} (\mu) \ = \ \inf_{u \in H^1 (\R), \, \int_\ell |u|^2 = \mu} E_\alpha (u, \ell).
\end{equation} 
The problem of the Ground States for $E_\alpha (\cdot, \ell)$ was
studied in \cite{FOO-08,FJ-08}. For an exhaustive review see \cite{T-23}.


Concerning the term $E_\sigma (v, \Pi)$ in \eqref{e},
we introduced the strength of the contact interaction 
$\sigma$ with either sign, restricted to 
the focusing nonlinearity by choosing a negative sign for the last term, and
moved from the physical fourth power to a generic power $r$. In fact, in this paper we do not allow
the values $r \geq 4$ because 
$r=4$ is the critical power
in dimension two, as described by the two-dimensional Gagliardo-Nirenberg inequality:
$$ \| v \|_{L^r (\R^2)}^r \ \leq \ C \| \nabla v \|_{L^2 (\R^2)}^{r-2} \| v \|_{L^2(\R^2)}^2, $$
that states that from $r = 4$ the kinetic term in $E_\sigma (\cdot, \Pi)$ ceases to control the nonlinear
term. 
From the construction of the
two-dimensional point interaction follows that in order to recover the standard focusing two-dimensional NLS energy, one has to set $\sigma = \infty$ instead of the more 
intuitive $\sigma = 0$. This delicate point can be heuristically 
explained by considering that for $\sigma = \infty$ in order
to get a finite energy it must be $q = 0$, that reduces the 
domain $\D_{2}$ to $H^1 (\R^2)$ through $v = \phi$ and makes the energy 
$E_\sigma (\cdot, \Pi)$ equal to
$$ E_\infty (v, \Pi) \ = \ \frac 1 2 \| \nabla v \|_{L^2 (\R^2)}^2
- \frac 1 r \| v \|_{L^r (\R^2)}^r,$$
that is the standard NLS energy.

We will use the symbol
$$\E_{\sigma, \Pi} (\mu) \ = \ \inf_{v \in \D_2, \int_\Pi |v|^2 = \mu} E (v, \ell).$$
The problem of the Ground States for $E_\sigma (\cdot, \Pi)$ was analyzed in \cite{ABCT-22} and in \cite{FGI-22} for the action functional of the 
same model.


Concerning the term $E_{\ju} (U)$ in \eqref{e}
we point out that considering a non-negative strength $\beta$ is not restrictive, 
since for every complex $\beta$ it is possible to reduce to \eqref{ebeta} by suitably modulating the relative phase between $q$ and $u (0)$. The negative sign corresponds to the choice for which having the same constant phase on $\ell$ and $\Pi$ is 
energetically convenient.


\subsection{Formulation of the problem}
As anticipated, the problem we investigate is the existence of Ground States. Among the several notions of Ground States present in
the literature and consistently with the physical
features of the Ground States in a condensate, we choose the one which generalizes the ordinary quantum
mechanical notion of Ground State to a nonlinear setting.
\begin{definition}
Given $\mu > 0$, a Ground State at mass $\mu $ for the energy functional $E$ defined by \eqref{e} on the domain $\D \subset L^2 (\J)$, is an element of $\D$ that satisfies the
mass constraint
\begin{equation}
\label{mass}
\int_\J |U|^2 \ = \ \int_\ell |u|^2 + \int_{\Pi} |v|^2 \ = \ \mu \ 
\end{equation}
and minimizes $E$ among all functions in $\D$ satisfying the same constraint.
\end{definition}

It is convenient to introduce the notation
\begin{equation} \label{dmu}
{\mathcal D}^\mu : = \{ U \in {\mathcal D}\, {\rm{s.t.}} \, U \, {\rm{fulfils \ the \ constraint \ \eqref{mass}}} \}
\end{equation}
and to denote
\begin{equation}
\label{notinf}
\E (\mu) \ = \ \inf_{U \in D^\mu} E (U).
\end{equation}


We are then led to the following formulation of
the problem of the Ground States for the Nonlinear
Schr\"odinger Equation on the hybrid $\J$:

\medskip

\noindent 
{\em
Let $\alpha, \sigma \in \R$, $\beta \geq 0$, $2 < p < 6$, $2 < r < 4$, and  $\mu > 0$. Establish whether there
exist Ground States at mass $\mu$ of the energy \eqref{e}.}

\medskip


Notice that without imposing the mass constraint there would be no Ground State, since $E (\lambda U) \to - \infty$ as $\lambda \to \infty$ for every non-zero $U \in \D$.
However the presence of the constraint alone does not ensure the existence for the following  reasons: 
\begin{itemize}
\item
first, the energy, although constrained
to a fixed mass, may not be lower bounded. This happens for instance in the problem on the line $\ell$ in the supercritical case  $p > 6$ for every value of the mass. To realize it, consider a function $u$ of a real variable and the family of functions $u_\lambda (x) = \sqrt \lambda u (\lambda x).$ Then,
$$ E_\alpha (u_\lambda, \ell) \ = \ \frac{\lambda^2}{2} \| u' \|_{L^2 (\R)}^2
- \frac{\lambda^{\frac p 2 -1}} p \| u\|_{L^p (\R)}^p + 
\frac {\lambda^2 \alpha} 2 | u (0) |^2 \ \to \ - \ \infty, \quad \lambda \to + \infty. $$
Therefore, since $\| u_\lambda \|_{L^2 (\R)} = \| u \|_{L^2 (\R)}$
we have that the energy $E_\alpha (\cdot, \ell)$ constrained to the arbitrary mass $\| u \|_{L^2 (\R)}^2 > 0$ is not lower bounded. 
Notice that the assumption $p > 6$
is crucial to this conclusion;

\item on the other hand, consider the subcritical case $p < 6$ on the line; then the Gagliardo-Nirenberg  inequality \eqref{gn1} applied to $E_\alpha (\cdot, \ell)$ with $\alpha > 0$ gives the uniform lower bound
\begin{equation}
E_\alpha (u, \ell) \ \geq \ \frac 1 2 \| u' \|_{L^2 (\R)}^2
- C \| u' \|_{L^2 (\R)}^{\frac p 2 -1}
\mu^{\frac p 4 + \frac 1 2} \ \geq \ M \ > \ -\infty,
\end{equation}
where we crucially used $p < 6$ and denoted $\mu = \| u \|^2_{L^2 (\R)}$. Therefore, owing to the constraint  the
energy functional $E_\alpha (\cdot, \ell)$ is lower bounded
for every mass, so that
$$ \inf_{u \in H^1 (\R), \int_\ell |u|^2 = \mu} E_\alpha (u, \ell) > - \infty.$$
On the other hand, for any $u \in H^1 (\R)$ such that $\int_\R |u|^2 dx = \mu$, due to the choice $\alpha > 0$ it holds
\begin{equation} \label{notattained}
 E_\alpha (u, \ell) \ \geq \ E_0 (u, \ell) \ \geq \ E_0 ( \varphi_\mu, \ell), 
\end{equation}
where $\varphi_\mu$ is the
soliton introduced in \eqref{soliton}. One then gets
\begin{equation}
 \mathcal E_{\alpha, \ell} (\mu) \geq \mathcal E_{0, \ell} (\mu). 
\label{infuno}
\end{equation}
On the other hand, consider the family of translations 
$\varphi_{\mu,\lambda} (x) \ = \ \varphi_\mu (x - \lambda)$
of the Ground State $\varphi_\mu$ at mass $\mu$ for $E_0 (\cdot, \R)$, so
$$E_\alpha (\varphi_{\mu, \lambda}, \ell) \ = \ E_0 (\varphi_\mu, \ell) + \frac \alpha 2 | \varphi_\mu (\lambda) |^2 \ \to \ 
E_0 (\varphi_\mu, \ell), \quad \lambda \to \infty,
$$
\begin{equation} \label{infdue}
\mathcal E_{\alpha, \ell} (\mu) \leq \mathcal E_{0, \ell} (\mu). 
\end{equation}
Then, from \eqref{infuno} and \eqref{infdue} one concludes
\begin{equation} \label{inftre}
\mathcal E_{\alpha, \ell} (\mu) \ = \ \mathcal E_{0, \ell} (\mu).
\end{equation}

Now, the first inequality in \eqref{notattained} is an equality if and only if $u (0) = 0$, while the second is an equality if and only if $u = \varphi_\mu$ 
up to translations and multiplications by a phase. But as $\varphi_\mu$ is strictly positive, the two inequalities cannot be both equalities, so that
\begin{equation} \label{strictlyworse}
E_\alpha (u, \ell) \ > \ E_0 (\varphi_\mu, \ell), \quad \forall u \in H^1 (\R) \, \ {\rm s.t.} \, \int_\R |u|^2 = \mu,
\end{equation}
from which
it follows that there is no function $u$ that attains the infimum $\mathcal E_{\alpha, \ell} (\mu)$, so that there is no ground state at mass $\mu$.
\end{itemize}

Summing up, the existence of a Ground State at mass $\mu$ 
is equivalent to two conditions: first, the infimum of the constrained energy must be finite. Second, it must be attained by some function. We stress that, contrarily to what happens in 
standard Quantum Mechanics, when a nonlinearity is involved these
two conditions are not guaranteed.

\section{The main result}
Here we give a condition equivalent to the existence of Ground States
at mass $\mu$. For the proof we refer to the proof of the
analogous result given in \cite{ABCT-24} for the hybrid plane $\I$.

\subsection{Lower boundedness of the constrained energy}
First,
like in the case of classical domains, the lower boundedness 
of the constrained energy is guaranteed by subcriticality through
the Gagliardo-Nirenberg estimates, which establish that every
subcritical nonlinearity can be controlled by the kinetic term of the energy. Now, such estimates are
widely known for functions in Sobolev spaces, but here one
needs them to hold in the energy domain $\D$, that strictly
contains the space $H^1 (\R) \oplus H^1 (\R^2)$. 
In \cite{ABCT-24} it was shown how to extend 
Gagliardo-Nirenberg estimates to the hybrid plane $\I$, and the very same method can be applied for the hybrid manifold $\J$.

First of all notice that, in the nonlinear term 
$$ - \frac 1 p \| u \|_{L^p (\R)}^p - \frac 1 r \| v \|_{L^r (\R^2)}^2, $$
the first term can be treated by the classical Gagliardo-Nirenberg
estimate \eqref{gn1}. The same is not true for the second term,
which lies outside $H^1(\R^2)$ unless $q = 0$. Following \cite{ABCT-22,ABCT-24}, for the case $q \neq 0$ we rewrite $v \in \D_2$, $v = \phi + q \frac {K_0}{2 \pi},$ as
\begin{equation} \label{decq}
v (x) \ = \ \phi_{q} (x) + q \frac{K_0 (|q|x)}{2 \pi}, \qquad
\phi_q \in H^1 (\R^2),
\end{equation}
which is possible since $K_0 (|q|x) \sim K_0 (x) \sim - \log |x|$ as $x \to 0$. Then
\begin{equation} \label{diffgreen}
\phi_q (x) = \phi (x) + \frac q {2 \pi}
(K_0 (x) - K_0 (|q| x)).
\end{equation}


One has the following proposition:

\begin{proposition}[Gagliardo-Nirenberg estimate] \label{prop:gnj}
Let $U = (u,v) \in \D^\mu$. Write 
\begin{equation} 
\nonumber
v (x) \ = \ \left\{
\begin{array} {cc}
\phi_q (x) + q \frac {K_0 (|q|x)}{2 \pi}, & \ q \neq 0 \\
\phi_0 (x), & \ q = 0 \end{array}
\right.
\end{equation}
Then,
\begin{equation}
\label{gnj} \begin{split}
\| u \|_{L^p (\R)}^p \  & \leq \ C \| u' \|_{L^2 (\R)}^{\frac p 2 -1}
\\
\| v \|_{L^r (\R^2)}^r \ & \leq \ C 
( \| \nabla \phi_q \|_{L^2 (\R^2)}^{r-2} + |q|^{r-2} ),
\end{split} \end{equation}
where the mass $\mu$ was absorbed in the constants $C$.

\end{proposition}
As in the standard case, Gagliardo-Nirenberg estimates are
useful in the proof of the lower boundedness of the constrained energy. To show this, we first rewrite the energy of the planar 
component $v$ in terms of the decomposition \eqref{decq}. 
We get the following result.
\begin{proposition}
\label{prop:e2q}
Let $v \in \D_2$ with $q \neq 0$ and consider the decomposition \eqref{decq}. Then
\begin{equation} \label{e2q}
E_{\sigma}(v, \Pi) \ = \ \frac 1 2 \| \nabla \phi_q \|_{L^2 (\R^2)}^2 + \frac{| q |^2} 2 \| \phi_q \|_{L^2 (\R^2)}^2 -
\frac{|q|^2} 2  \| v \|_{L^2 (\R^2)}^2 + \left( \frac \sigma 2 
+ \frac{\log |q|}{4 \pi} \right) |q |^2 - \frac 1 r \| v \|_{L^r
(\R^2)}^r.
\end{equation}
\end{proposition}

\begin{proof}
We define the Fourier transform as
\begin{equation} \label{fourier}
\widehat f (k) \ = \ \frac 1 {2 \pi} \int_{\R^2} f (x) e^{-i k \cdot x} \, dx
\end{equation}
for $f \in L^1 (\R^2)$, then extend it to $L^2 (\R^2)$ by continuity. With such a definition the Fourier transform is a unitary 
operator in all Sobolev spaces $H^s (\R^2), \ s \in \R$.
 Now, by properties of the Bessel functions it is known that, in the space of distributions,
$$ ( - \Delta + 1) K_0 \ = \ 2 \pi \delta$$
and, since in Fourier space the gradient is mapped into the multiplication by the vector $ik$, one has
$$\widehat K_0 (k) = \frac 1 {k^2 + 1} \, \qquad k^2 := |k|^2,$$ then 
by \eqref{diffgreen}
one gets
\begin{equation} \label{diffgreenf}
\widehat \phi (k) = \widehat \phi_q (k) + \frac q {2 \pi}
\frac{1 - |q|^2}{(k^2 + |q|^2)(k^2 +1)}.
\end{equation}
The two first terms of \eqref{esigma} can be rewritten as
\begin{equation} \begin{split} 
&    \frac 1 2 \| \phi \|_{H^1 (\R^2)}^2 - \frac 1 2 \| v \|_{L^2 (\R^2)}^2 \\
  = & \frac 1 2  \int_{\R^2} (k^2 +1 ) \left| \widehat \phi_q (k) + \frac q {2 \pi}
\frac{1 - |q|^2}{(k^2 + |q|^2)(k^2 +1)} \right|^2 dk - \frac 1 2
\int_{\R^2} \left|
\widehat \phi_q (k) + \frac q {2 \pi (k^2 + |q|^2)}
\right|^2 dk \\
 = & \frac 1 2 \int_{\R^2} k^2 | \widehat \phi_q (k)|^2 dk
- \Re \int_{\R^2} \frac{|q|^2 q \, \overline{\widehat \phi_q (k)}}
{2 \pi (k^2 + |q|^2)} dk + \frac{|q|^2}{8 \pi^2}
\int_{\R^2} \frac{dk}{k^2 + |q|^2} \left(\frac{(1 - |q|^2)^2}
{k^2 +1} -1 \label{twofirst}
\right)
\end{split} \end{equation}

On the other hand, the three first terms in \eqref{e2q} can be
rewritten as
\begin{equation} \begin{split} \label{threefirst}
& \frac 1 2 \| \nabla \phi_q \|_{L^2 (\R^2)}^2 + \frac{|q|^2} 2
\| \phi_q \|_{L^2 (\R^2)}^2 - \frac {|q|^2} 2 \| v \|_{L^2 (\R^2)}^2
\\
= & \frac 1 2  \int_{\R^2} (k^2 + |q|^2) | \phi_q (k)|^2 dk
- \frac {|q|^2}{2} 
\int_{\R^2} \left|
\widehat \phi_q (k) + \frac q {2 \pi (k^2 + |q|^2)}
\right|^2 dk \\
= & \frac 1 2  \int_{\R^2} k^2 | \phi_q (k)|^2 dk - |q|^2 \Re \, q
\int_{\R^2} \frac{\overline{\widehat \phi_q (k)} \, dk}
{2 \pi (k^2 + |q|^2)} - \frac{|q|^4}{8 \pi^2} \int_{\R^2}
\frac{dk}{(k^2 + |q|^2)^2}.
\end{split}
\end{equation}
Then, taking the difference between the quantities computed in \eqref{twofirst} and \eqref{threefirst} we conclude
\begin{equation} \label{difference} \begin{split}
& \frac 1 2 \| \phi \|_{H^1 (\R^2)}^2 - \frac 1 2 \| v \|_{L^2 (\R^2)}^2 - \left(\frac 1 2 \| \nabla \phi_q \|_{L^2 (\R^2)}^2 + \frac{|q|^2} 2
\| \phi_q \|_{L^2 (\R^2)}^2 - \frac {|q|^2} 2 \| v \|_{L^2 (\R^2)}^2 \right) \\
= & \ \frac{|q|^2}{8 \pi^2} \int_{\R^2} \frac{dk}{(k^2 + |q|^2)^2}
\left( \frac {(1 - |q|^2)^2}
{k^2 +1} -1 + |q|^2
\right) \\
= & \ \frac{|q|^2}{8 \pi^2} \int_{\R^2} \frac{|q|^2 -1 }{(k^2 + |q|^2)(k^2 + 1)} dk \ = \ \frac{|q|^2}{4 \pi} \log |q|.
\end{split}
\end{equation}
Therefore, from \eqref{esigma} and \eqref{difference}
\begin{equation}
E_\sigma (v, \Pi) \ = \ \frac 1 2 \| \nabla \phi_q \|_{L^2 (\R^2)}^2 + \frac{|q|^2} 2
\| \phi_q \|_{L^2 (\R^2)}^2 - \frac {|q|^2} 2 \| v \|_{L^2 (\R^2)}^2 + 
\frac{|q|^2}{4 \pi} \log |q| + \frac \sigma 2 |q|^2 - \frac 1 r
\| v \|_{L^2 (\R^2)}^2,
\end{equation}
and the proof is complete.
\end{proof}
Plugging estimate \eqref{gnj} in the expression of the energy \eqref{e} one has the following
\begin{proposition} \label{prop:lowerb}
Let $2 < p < 6$ and $2 < r < 4$. Then the functional \eqref{e}
with the constraint \eqref{mass} is lower bounded.
\end{proposition}
 The proof given in \cite{ABCT-24} for the hybrid plane $\I$ can be closely repeated for $\J$. Here we recall the main passage. After estimating the coupling term by
$$ - \beta \Re \, (\overline q u (0) ) \ \geq \ - C (|q|^2 + \sqrt \mu \| u' \|_{L^2 (\R)} ), $$
through \eqref{gnj} we have 
\begin{equation} \label{lowerboundenergy}
\begin{split}
E (U)  \geq & \, \frac{1}{2}\|u'\|_{L^2 (\R)}^2-C \left(\| u' \|_{L^2 (\R)}^{\frac p 2 -1}+ \| u' \|_{L^2 (\R)}\right) + 
    \frac{1}{2}\| \nabla \phi_{q} \|^2_{L^2 (\R^2)} \\[.2cm]   & \,  -C
    \| \nabla \phi_{q} \|^{r-2}_{L^2 (\R^2)}
 + \frac{|q|^2}{2} (\log |q | - C) - C{|q|^{r-2}},
\end{split}
\end{equation}
where the mass was absorbed in the constants and the contributions of $\| u' \|_{L^2 (\R)}, \| \nabla
\phi_q \|_{L^2 (\R^2)}$, and $q$ are separated from one another, so that one can introduce three functions $f,g,h$ that are
lower bounded due to the subcritical character of the
nonlinearities. Then one gets
\begin{equation} \label{fgh}
E (U) \ \geq \ f ( \| u' \|_{L^2 (\R)}) + g
( \| \nabla \phi_q \|_{L^2 (\R^2)}) + h (|q|),
\end{equation}
and it
follows that the constrained energy is lower bounded too. 

\subsection{Minimizing sequences}
Let $\{U_n \} \subset D^\mu$ be a minimizing sequence for the 
energy $E$ at mass $\mu$. This means that every element of the sequence has mass $\mu$ and that
$$ \lim E (U_n) = \E (\mu). $$
We aim at singling out the Ground States as limits of the minimizing sequences.

Writing $U_n = (u_n,v_n)$
and
considering the decomposition
\begin{equation}
    \label{decqn}
   v_n (x) = \phi_{n,q_n} (x) + q_n 
   \frac{K_0
   ( |q_n| x) } {2 \pi} 
\end{equation}
by \eqref{fgh} one has
$$E (U_n) \geq f ( \| u_n' \|_{L^2 (\R)})
+ g ( \| \nabla \phi_{n,q_n} \|) + h (|q_n|).
$$
Now, since $E(U_n)$ is a convergent real
sequence, it has to be bounded, so none of the three terms can diverge, which implies that
the three quantities $ \| u_n' \|_{L^2 (\R)},
 \| \nabla \phi_{n,q_n} \|, |q_n|$ are bounded too. As a consequence, $u_n$ is a bounded
 sequence in $H^1 (\R)$, $\phi_{n,q_n}$ is a bounded sequence in $H^1 (\R^2)$, $q_n$ is a numerical bounded sequence. 
 
 \begin{remark} \label{convergenze}
 Therefore, up to
 subsequences,
 \begin{itemize}
     \item By Banach-Alaoglu theorem $u_n$ converges to some $u$ weakly  in $H^1 (\R)$. 
     \item $q_n$ converges to some $q$ in $\C$.
     \item $q_n\frac{K_0
   ( |q_n| x) } {2 \pi}$ converges strongly to
   $q   \frac{K_0
   ( |q| x) } {2 \pi} $ in $L^p (\R^2)$ for $2 \leq p < \infty$ if $q \neq 0$. It converges to zero weakly in $L^2 (\R^2)$ if
   $q = 0$.
   
      \item By Banach-Alaoglu theorem 
      $\phi_{n,q_n}$ converges to some $\phi_q$ weakly  in $H^1 (\R^2)$. From this one can easily see
      that $\phi_n$, that appears in the decompositon \eqref{energydom2d} of $v_n$, converges
      weakly in $H^1 (\R^2)$ to the function $\phi$ that
      decomposes $v$ according to \eqref{energydom2d}, namely
      $$v = \phi + q \frac{K_0}{2\pi}.$$
      \item $v_n$ converges weakly to some $v$ in $L^p (\R^2)$, $2 \leq p < \infty$.
      \end{itemize}
      \end{remark}
As a consequence, $U_n$ converges weakly to 
$U$ in $L^2 (\J)$. Now, $U$ may not belong to the
minimizing domain $\D^\mu$ since, as a weak limit, its mass could be strictly smaller 
than $\mu.$ However, it turns out that, if the mass
is conserved in the weak limit so that $U \in \D^\mu$,
then $U$ is a Ground State.
\begin{proposition}
\label{massconserved}
Let us consider a minimizing sequence $\{ U_n \} \subset \D^\mu$, and denote by $U$ a weak limit of it in the
sense specified by Remark \ref{convergenze}.

If $\| U \|_{L^2 (\J)}^2 = \mu,$ then $U$ is a Ground State at mass $\mu$
for the energy functional $E$.
\end{proposition}

\begin{proof}
By weak convergence,
\begin{equation}
    \label{strongU}
    \| U_n - U \|_{L^2 (\J)}^2 = 2 \mu - 2 \Re (U_n,U)_{L^2 (\J)} \to 0,
\end{equation}
so conservation of mass in the weak limit ensures strong
convergence. This immediately entails that $u_n$, $v_n$ and $\phi_n$
converge strongly to $u$, $v$ and $\phi$ in the respective
$L^2$ spaces, while, as already seen, $q_n$ converges to $q$. Now, since Banach-Steinhaus' Theorem
guarantees boundedness in $L^2$ for the weakly convergent sequences $\{ u'_n \}, \ \{ \nabla \phi_n \},$ then  
by Gagliardo-Nirenberg estimates one gets
\begin{equation}
    \label{strongu}
    \begin{split}
\| u_n - u \|_{L^p (\R)}^p \ \leq & \ C \| u_n' - u' \|_{L^2 (\R)}^{\frac p 2 - 1} \| u_n - u \|_{L^2 (\R)}^{\frac p 2 + 1} 
\end{split}
\end{equation}
thus $u_n$ converges strongly to $u$ in $L^p (\R)$. On the other
hand 
\begin{equation}
    \label{strongu}
    \begin{split}
\| v_n - v \|_{L^r (\R^2)}^r \ \leq & \ C 
\| \phi_n - \phi \|_{L^r (\R^2)}^r + C |q_n -q |^r \\
\leq & \ C \| \phi_n - \phi \|_{L^2 (\R^2)}^2 \| \nabla \phi_n 
- \nabla \phi \|_{L^2 (\R^2)}^{r-2} + C |q_n - q|^r
\end{split}
\end{equation}
thus $v_n$ 
converges strongly to  $v$ in $L^r (\R^2)$. Therefore, the nonlinear terms in $E (U_n)$ converge to the corresponding terms in $E (U)$. 
All other terms converge except the two kinetic terms
$\frac 1 2 \| u_n' \|^2_{L^2 (\R)}$ and 
$\frac 1 2 \| \nabla \phi_n \|^2_{L^2 (\R^2)}$.
Then, since $U \in D^\mu$,
one has the following double inequality
\begin{equation}
    \label{inequality} \begin{split}
0 \ \geq \ & \E (\mu) - E (U) \ = \  \lim E (U_n) - E (U) \\
= \ & \frac 1 2 (\lim (\| u_n '\|_{L^2 (\R)}^2 +  \| \nabla \phi_n \|_{L^2 (\R^2)}^2 )- \| u \|_{L^2 (\R)}^2
- \| \nabla \phi \|_{L^2 (\R^2)}^2 )\\ 
    \geq \ & 0
\end{split}
\end{equation}
by weak convergence.
\ Then
$ \E (\mu) \ = \ E (U)$
so $U$ is a Ground State at mass $\mu$ and the proof is 
complete.

       \end{proof}
       Now it is necessary to give conditions that guarantee that
       the weak limit of a minimizing sequence preserves the
       mass, so to apply Proposition \ref{massconserved}. It turns out that there is a simple necessary and sufficient condition.

\begin{theorem} \label{thm:key}
A Ground State at mass $\mu$ for the energy functional $E (U)$ defined in \eqref{e} 
exists if and only if there is a function $U \in \D^\mu$ such that
\begin{equation} \label{ineq:key}
 E (U) \ \leq \ \mathcal E_{0,\ell} (\mu). 
\end{equation}
\end{theorem}
Such a result is relevant for two reasons: first, it
provides a  practical criterion for proving the existence of a Ground State by exhibiting a function $U$ that fulfils \eqref{ineq:key}; second, it states that the infimum of the constrained energy
can be reached in two ways only: either by constructing a Ground
State or by escaping at infinity through the line. This alternative greatly simplifies the options foreseen by the concentration-compactness theory. Let us further examine this point.

According to the concentration-compactness theory (see e.g. \cite{C-03}), a weakly convergent sequence can preserve its mass or lose it. Furthermore, the loss can be partial or complete. For
the present model, we proved in Proposition \ref{massconserved} that if for a given minimizing
sequence there is no loss of mass,
then the sequence converges to a Ground State. One can show that the partial loss of mass is never
energetically convenient due to the subadditivity of the
energy, so that it never applies to
minimizing sequences (see the proof of Theorem 1 in
\cite{ABCT-24}). The complete loss of mass can happen
through three mechanisms: dispersion of the sequence, with
vanishing $\| u_n \|_{L^p (\R)}$ and $\| v_n \|_{L^r (\R^2)}$, escape at infinity through the plane or escape at infinity through the line. It turns out that (\cite{ABCT-24}):
\begin{enumerate}
    \item 
if among the three mechanisms involved by
the complete loss of mass the most convenient is the
dispersion, then  
$$E (U_n) \to E_{\rm{lin}} (\Psi_\mu) = 
E (\Psi_\mu) + \frac 1 p \| \psi_\mu \|^p_{L^p (\R)} + \frac 1 r 
\| \varsigma_\mu \|^r_{L^r (\R^2)} ,
$$
 with $\Psi_\mu = (\psi_\mu, \varsigma_\mu)$ the  Ground State at mass $\mu$ of the quadratic part $E_{\rm{lin}}$ of $E.$
Therefore, $\E (\mu) = E_{\rm{lin}} (\Psi_\mu)$ and
$$
\E (\mu ) \ \leq \ E (\Psi_\mu) \ = \  E_{\rm{lin}}
(\Psi_\mu) - \frac 1 p \| \psi_\mu \|^p_{L^p (\R)} - \frac 1 r 
\| \varsigma_\mu \|^r_{L^r (\R^2)} \ < \ \E (\mu), $$
which is absurd. So there cannot be complete loss of mass
by dispersion;
\item if the most convenient mechanism is the escape at infinity through the plane, then
$E (U_n) \to E_\infty (\xi_\mu, \Pi)$, where $\xi_\mu$ is the soliton of the standard nonlinear Schr\"odinger equation in dimension two. But in this case
one has $\Xi_\mu = (0, \xi_\mu) \in \D^\mu$ and so $\Xi_\mu$ is a Ground State and one proves existence.
\item If, finally, the most convenient mechanism is the escape at infinity through the line, then
$E (U_n) \to E_0 (\varphi_\mu, \ell)$. Of course the state
$\Phi_\mu = (\varphi_\mu, 0)$ belongs to $\D^\mu$, but
$$ E (\Phi_\mu) \ = \ E_0 (\varphi_\mu, \ell) + \frac \alpha 2 | \varphi_\mu (0)|^2,$$
thus if $\alpha > 0$, then the escape at infinity is a convenient choice for the soliton.
\end{enumerate}
The previous discussion is aimed at clarifying why the only
hindrance to the existence of a Ground State is the possibility,
for a minimizing sequence, to run to infinity through the line.
In other words, the plane is not dangerous since it can always host the two-dimensional soliton centered at the origin, equalling the 
optimal energy $E_\infty (\xi_\mu, \Pi)$ reached when escaping through the plane.
In fact, it has been proved in \cite{ABCT-22} that on the plane it
is possible to make better: for every $\sigma \in \R$ there exist
a Ground State at mass $\mu$ whose energy is strictly less than
$E_\infty (\xi_\mu, \Pi)$. Then one can say that the plane is strictly
 attractive, regardless of the sign of $\sigma$. An analogous
result for the analogous three-dimensional problem
was proved in \cite{ABCT-22bis}.

Gathering together the previous comments, one ends up with
the following
\begin{corollary} \label{thm:convenient}
Inequality \eqref{ineq:key} and therefore the existence of a Ground State at mass $\mu$ for 
the energy functional \eqref{e} is guaranteed by one of the
following conditions:
\begin{enumerate}
\item The contact interaction between the plane and the line
is non-repulsive on the line, i.e. $\alpha \leq 0$. This entails that a soliton at the infinity on the line is not less energetic than
the same soliton centred at the origin, and therefore the escape 
at infinity along the line is not strictly convenient.
\item Given the mass $\mu$, the nonlinearity powers $p$ and $r$ are such that the plane is energetically more convenient than the line, in the sense that
$$ E_\sigma (\varsigma_\mu, \Pi) \ \leq \ E_\alpha (\varphi_\mu,
\ell). $$
More precisely, this is accomplished by the following  condition:
there exists a threshold value $\mu^\star$ for the mass, that depends on $p$ and $r$, such that
\begin{equation} \begin{split}
{\rm{if \ \,}} \mu > \mu^\star & \quad {\rm and} \quad {\frac 2 {4-r}} \ > \ {\frac{p+2}{6-p}} \\ or \\
{\rm{if \ \,}} \mu < \mu^\star & \quad {\rm and} \quad {\frac 2 {4-r}} \ < \ {\frac{p+2}{6-p}}, 
\end{split} \end{equation}
then there is a Ground State at mass $\mu$ for \eqref{e}.

\item The contact interaction between the plane and the line is attractive enough on the plane, which means that
$\sigma$ lies below a threshold $\sigma^\star$. Indeed, since
$\E_{\sigma,\Pi}$ is a monotonically increasing
function of $\sigma$ and
$$ \E_{\sigma^\star,\Pi} (\mu) \to - \infty, \qquad \sigma \to - \infty, $$
then either
$$ \E_{+\infty,\Pi} (\mu) \ < \ \E_{0, \ell} (\mu)$$
or
there exists $\sigma^\star \in \R$
such that
$$\E_{\sigma^\star,\Pi} \ = \ \E_{0, \ell} (\mu),$$
Then, for $\sigma \leq \sigma^\star$ possibly infinite, there exists a Ground State.

\item The junction energy $E_{\mathcal J}$ is strong enough, namely $\beta \geq \beta^\star$ for a certain threshold $\beta^\star$.
In this case, $E_\J$ is a monotonically decreasing function of
$\beta$, so one has
$$\E (\mu) \to - \infty, \qquad \beta \to + \infty$$
and then the result immediately follows. 
\end{enumerate}

\end{corollary}
Notice that the previous Corollary makes more precise the prescriptions
given in Section 1.1 in order to construct a trapping hybrid plane.

We point out that the criterion given at (1) is simpler than the analogous condition given in \cite{ABCT-22} for the hybrid $\I$. 
This reflects the fact that the problem of the existence of Ground States on a halfline with a delta interaction at the origin is
more complicated than the same problem on the line, as 
detailed in \cite{BC-23}.

\subsection{Further results}
We list here for the sake of completeness some results that immediately translate from \cite{ABCT-24} to this context.
\begin{enumerate}
\item A Ground State always exists for small mass.
Indeed, the linear ground state $\Psi_\mu$ is
energetically more convenient than the one-dimensional soliton,
as it scales linearly with the mass while $\E_{0,\ell}$ scales superlinearly. 
\item Fixed $\mu > 0$, if $\alpha$ and $\sigma$ are large enough, while $\beta$ is
small enough, then there is no Ground State. Here we are in fact
denying the hypotheses of existence theorem. Notice that 
the threshold depends on $\mu$.
\item If $\beta = 0$ and there exists a Ground State, then it
is supported either on $\ell$ or on $\Pi$ and it coincides 
with a Ground State of the problem on $\ell$ or on $\Pi$.

\item If $\beta \neq 0$, then every Ground State $U = (u,v)$ has
non-trivial components $u$ and $v$. Moreover, $u \in H^2 (\R \backslash \{0 \}) \cap H^1 (\R)$, while $v = \phi + q \frac{K_0}{2\pi}$ with 
$\phi \in H^2 (\R^2).$ More precisely, $u$ is an even function
made by glueing together two chunks of a soliton, while $v$ is 
a radially symmetric function whose shape is that of a Ground State
for the Nonlinear Schr\"odinger Equation on the plane with a point interaction at the origin. They are connected by the matching
conditions
\begin{equation} \label{bc}
\left\{ \begin{array}{ccc} u' (0+) - u' (0') & = & \alpha u (0)
- \beta q \\ \\ \phi (0) & = & - \beta u (0) + \sigma q
\end{array}
\right.
\end{equation}
\end{enumerate}

\begin{remark}
In our search for Ground States we excluded the cases
$\alpha = \infty$ and $\sigma = \infty$. This was done because
such cases lead to modifications of the functional $E$ and
of the energy domain $\mathcal D$ and including them would make
the formulation of the problem too cumbersome. However, since they are not completely trivial, we quickly summarize the results.

The condition $\alpha = \infty$
amounts to imposing Dirichlet boundary conditions at the
origin of the line. As a consequence,
both the contact interaction on the line and the coupling term vanish. Therefore one has
$$ E (U) \ = \ E_{\infty} (u, \ell) + E_{\sigma} (v, \Pi)
\ = \ E_0 (u, \ell) + E_\sigma (v, \Pi)$$ on the domain
$$ {\mathcal D}_{\rm{dir}} = \{ U = (u,v), \, u \in H^1 (\R) \ {\rm{s.t.}} \ u (0) = 0, \ v \in \mathcal D_2 \}.
$$
The problem reduces then to a competition between $\ell$ and $\Pi$. If $p,r, \sigma$ and $\mu$ are such that the line turns out to be more convenient, then minimizing sequences leave the plane and seek the lowest energy on the line, which is given by the energy
of the soliton that cannot be attained beacuse of the Dirichlet boundary condition at the origin. Then there is no Ground State. 
If vice versa
the chosen parameters make the plane energetically 
convenient with respect to the line, then minimizing sequences 
concentrate on the plane and converge to the two-dimensional
Ground State according to the results of \cite{ABCT-22}.

In the case with $\sigma = \infty$ the line and the plane are 
decoupled too since, as discussed in  \ref{sec:comments},
such conditions entail $q=0$. Therefore minimizing sequences 
concentrate on the most convenient component among $\ell$ and $\Pi$. In the first case, the issue of the existence of a Ground State reduces to that of the existence for $E_\alpha (u, \ell)$. In the second case, a Ground State exists.

If $\alpha = \sigma = \infty$, then the existence of Ground States is determined by the competition between the free line and the free plane.
\end{remark}

Lastly, we stress that the energy functional \eqref{e} does
not exhaust all possible choices for the contact interactions.
In fact, if one starts from the linear case and seeks for all
possible interactions taking place at the junction, then one is
led to consider a nine-parameter family of possible energy functionals. This
family includes several rich dynamics as the one generated by the so-called delta-prime interaction, that on $\ell$ would give
rise to singular phenomena like symmetry breaking bifurcation 
on the Ground State (\cite{AN-13,ANV-13}). We plan to investigate such dynamics in a 
forthcoming study. A further point that we left untouched concerns the uniqueness of Ground States. This is in general a delicate
point that requires ad hoc techniques. For instance, for the analysis on uniqueness on metric graphs instead of hybrids see \cite{DST-20}.

\bigskip
\noindent
{\bf Acknowledgements.}

F.B. and L.T. have been partially supported by the INdAM GNAMPA project 2023 \textquotedblleft Modelli nonlineari in presenza di interazioni puntuali" (CUP E53C22001930001).

R.A., R.C. and L.T. acknowledge that this study was carried out within the project E53D23005450006 \textquotedblleft
Nonlinear dispersive equations in presence of singularities" - funded by European Union - Next Generation EU within the PRIN 2022 program (D.D. 104 - 02/02/2022 Ministero dell’Università e della Ricerca).


\end{document}